\documentclass[a4paper]{article}

\usepackage[english]{babel}
\usepackage[utf8x]{inputenc}
\usepackage[T1]{fontenc}

\usepackage{fullpage}
\usepackage{amsmath}
\usepackage{amssymb}
\usepackage{amsthm}
\usepackage{mathabx}
\usepackage{graphicx}

\usepackage{hyperref}
\newcommand{\R}{\mathbb{R}}

\newcommand{\Sp}{\mathbb{S}}

\theoremstyle{definition}
\newtheorem{definition}{Definition}

\theoremstyle{plain}
\newtheorem{theorem}{Theorem}
\newtheorem{proposition}{Proposition}

\newtheorem{lemma}{Lemma}
\newtheorem{conjecture}{Conjecture}

\theoremstyle{remark}
\newtheorem{remark}{Remark}
\newtheorem{claim}{Claim}
\title{Vector-Valued Maclaurin Inequalities}
\author{Silouanos Brazitikos, Finlay McIntyre}

\begin{document}
	\maketitle
	
	\begin{abstract}
		We investigate a Maclaurin inequality for vectors and its connection to an Aleksandrov-type inequality for parallelepipeds.
	\end{abstract}
	
	\section{Introduction}
	The classical Maclaurin inequality compares consecutive symmetric sums for any sequence of positive real numbers:
	\begin{theorem}[Maclaurin inequality]\label{thm:maclaurin inequalities}
		For any sequence of positive real numbers $x_1,\dots,x_m$ and any $1 < k \le m$, the following inequality holds:
		\begin{align*}
			\left(\frac{\sum\limits_{1\le i_1 < \cdots < i_{k}\le m}x_{i_1}\cdots x_{i_{k}}}{\binom{m}{k}}\right)^{\frac{1}{k}} \le \left(\frac{\sum\limits_{1\le i_1 < \cdots < i_{k-1}\le m}x_{i_1}\cdots x_{i_{k-1}}}{\binom{m}{k-1}}\right)^{\frac{1}{k-1}} .
		\end{align*}
	\end{theorem}
	This was first proved in \cite{newton1761arithmetica} and
	a proof of this result using elementary methods can be found also in \cite{inequalities}. The Maclaurin inequality can be seen as a refinement of the arithmetic-geometric mean inequality by noting that the geometric mean and arithmetic mean appear as the smallest and largest quantities respectively in the following chain of inequalities:
	\begin{align*}
		\left(x_1\cdots x_m\right)^{\frac{1}{m}}\le \left(\frac{\sum\limits_{1\le i_1 < \cdots < i_{m-1}\le m}x_{i_1}\cdots x_{i_{m-1}}}{\binom{m}{m-1}}\right)^{\frac{1}{m-1}} \le \cdots \le 
		\left(\frac{\sum\limits_{1\le i <  j\le m}x_{i}x_{j}}{\binom{m}{2}}\right)^{\frac{1}{2}} \le \frac{\sum_{i=1}^mx_i}{m}, 
	\end{align*}
which follows directly from Theorem \ref{thm:maclaurin inequalities}.
	
	In this note, we explore a variant of this classical result, with sequences of numbers replaced by families of vectors, and the standard product replaced by the wedge product operator. More precisely, let $v_1,\dots,v_m \in \R^d$ with $d \le m$, and for any $1 \le i_1<\cdots<i_k \le m$ with $k\le d$ denote by $|v_{i_1}\wedge\cdots\wedge v_{i_k}|$ the $k$-dimensional volume of the parallelotope spanned by $v_{i_1},\dots,v_{i_k}$. We are interested in "vector-valued" inequalities of the form 
\begin{align}\label{ineq:vec maclaurin with p}
	\left(\frac{\sum\limits_{1\le i_1 < \cdots < i_{k}\le m}|v_{i_1}\wedge\cdots\wedge v_{i_{k}}|^p}{\binom{m}{k}}\right)^{\frac{1}{kp}} \le \left(\frac{\sum\limits_{1\le i_1 < \cdots < i_{k-1}\le m}|v_{i_1}\wedge \cdots\wedge v_{i_{k-1}}|^p}{\binom{m}{k-1}}\right)^{\frac{1}{(k-1)p}},
\end{align}
with $p\in [0,\infty]$ and $2\le k\le d$. Note that if $m=d$ and $v_1,\dots,v_m$ are orthogonal, then each term $|v_{i_1}\wedge\cdots\wedge v_{i_k}|^p$ will just be equal to a $k$-fold product of numbers, namely $\|v_{i_1}\|^p\cdots \|v_{i_k}\|^p$. It follows that (\ref{ineq:vec maclaurin with p}) reduces to a special case of the classical Maclaurin inequality for any $p\in(0,\infty)$. 

Given a general family of vectors $v_1,\dots,v_m \in \R^d$, the value of $p$ plays a more important role. Using elementary results from linear algebra, we are able to establish (\ref{ineq:vec maclaurin with p}) for $p=2$ and $m=d$:
\begin{theorem}\label{thm:vec maclaurin with squares}
	For any $d$-tuple of vectors $v_1,\dots,v_d \in \R^d$ and any $1 < k \le d$, the following inequality holds:
	\begin{align*}
		\left(\frac{\sum\limits_{1\le i_1 < \cdots < i_{k}\le d}|v_{i_1}\wedge\cdots\wedge v_{i_{k}}|^2}{\binom{d}{k}}\right)^{\frac{1}{k}} \le \left(\frac{\sum\limits_{1\le i_1 < \cdots < i_{k-1}\le d}|v_{i_1}\wedge \cdots\wedge v_{i_{k-1}}|^2}{\binom{d}{k-1}}\right)^{\frac{1}{k-1}} .
	\end{align*}
\end{theorem}
\noindent Moreover, taking limits as $p \to \infty$, we can write
 \begin{align*}
 	\lim\limits_{p \to\infty}\left(\frac{\sum\limits_{1\le i_1 < \cdots < i_{k}\le m}|v_{i_1}\wedge\cdots\wedge v_{i_{k}}|^p}{\binom{m}{k}}\right)^{\frac{1}{p}} = \max\limits_{1\le i_1 < \cdots < i_{k}\le m}|v_{i_1}\wedge\cdots\wedge v_{i_{k}}|. 
 \end{align*}
Again using purely linear algebra, namely Szasz's inequality for subdeterminants, we are able to prove the following endpoint case:
\begin{theorem}\label{thm:vec maclaurin with p = infty}
	Fix vectors $v_1,\dots,v_m \in \R^d$ with $1\le d \le m$. Then, for any $1< k\le d$, the following inequality holds:
	\begin{align}
		\left(\max\limits_{1\le i_1 < \cdots < i_{k}\le m}|v_{i_1}\wedge\cdots\wedge v_{i_{k}}|\right)^{\frac{1}{k}} \le\left(\max\limits_{1\le i_1 < \cdots < i_{k-1}\le m}|v_{i_1}\wedge\cdots\wedge v_{i_{k-1}}|\right)^{\frac{1}{k-1}}.
	\end{align}
\end{theorem}
\noindent By a similar argument, we also prove (\ref{ineq:vec maclaurin with p}) for $p=0$:
\begin{theorem}\label{thm:vec maclaurin with p = 0}
	Fix vectors $v_1,\dots,v_m \in \R^d$ with $1\le d \le m$. Then, for any $1< k\le d$, the following inequality holds:
	\begin{align*}
		\left(\prod\limits_{1\le i_1 < \cdots < i_{k}\le m}|v_{i_1}\wedge\cdots\wedge v_{i_{k}}| \right)^{\frac{1}{\binom{m}{k} k}} \le \left(\prod\limits_{1\le i_1 < \cdots < i_{k-1}\le m}|v_{i_1}\wedge\cdots\wedge v_{i_{k-1}}| \right)^{\frac{1}{\binom{m}{k-1} (k-1)}}.
	\end{align*}
\end{theorem}
\noindent It is not difficult to verify that (\ref{ineq:vec maclaurin with p}) fails to hold in general for negative values of $p$.

If we take $p=1$, it seems more difficult to establish the desired inequality for a general family of vectors. However, in the case where $m=d$, we have some partial results: 
\begin{theorem}\label{thm:vec maclaurin list of known p=1 cases}
	If $m=d$, then for any $v_1,\dots,v_m \in \R^d$, inequality (\ref{ineq:vec maclaurin with p}) holds with $p=1$ and $k=2,3,d$ in all dimensions $d$.
\end{theorem}
\noindent Using a certain duality between families of vectors, one can also prove the case for $p=1$ and $k=4$ in dimensions $5$ and $6$ - see \cite{finlays_thesis}, Section 6 for details. To prove Theorem \ref{thm:vec maclaurin list of known p=1 cases}, in each case we essentially construct a new family of orthogonal vectors $\tilde{v}_1,\dots,\tilde{v}_m \in \R^d$ such that
\begin{align*}
	\left(\frac{\sum\limits_{1\le i_1 < \cdots < i_{k}\le m}|v_{i_1}\wedge\cdots\wedge v_{i_{k}}|}{\binom{m}{k}}\right)^{\frac{1}{k}} \le
	 \left(\frac{\sum\limits_{1\le i_1 < \cdots < i_{k}\le m}|\tilde{v}_{i_1}\wedge\cdots\wedge \tilde{v}_{i_{k}}|}{\binom{m}{k}}\right)^{\frac{1}{k}},
\end{align*}
and
\begin{align*}
	\left(\frac{\sum\limits_{1\le i_1 < \cdots < i_{k-1}\le m}|v_{i_1}\wedge \cdots\wedge v_{i_{k-1}}|}{\binom{m}{k-1}}\right)^{\frac{1}{k-1}} \ge
	\left(\frac{\sum\limits_{1\le i_1 < \cdots < i_{k-1}\le m}|\tilde{v}_{i_1}\wedge \cdots\wedge \tilde{v}_{i_{k-1}}|}{\binom{m}{k-1}}\right)^{\frac{1}{k-1}}.
\end{align*}
The desired result then follows by a simple application of Theorem \ref{thm:maclaurin inequalities}.

In light of these results, we conjecture that the vector-valued Maclaurin inequalities should hold in the full range $0\le p \le \infty$:
\begin{conjecture}[Vector-valued Maclaurin inequality]\label{conj:vec maclaurin}
	Fix vectors $v_1,\dots,v_m \in \R^d$ with $1\le d \le m$. Then for all $p\in [0,\infty]$ and $2\le k\le d$, the following inequality holds:
	\begin{align*}
		\left(\frac{\sum\limits_{1\le i_1 < \cdots < i_{k}\le m}|v_{i_1}\wedge\cdots\wedge v_{i_{k}}|^p}{\binom{m}{k}}\right)^{\frac{1}{kp}} \le \left(\frac{\sum\limits_{1\le i_1 < \cdots < i_{k-1}\le m}|v_{i_1}\wedge \cdots\wedge v_{i_{k-1}}|^p}{\binom{m}{k-1}}\right)^{\frac{1}{(k-1)p}},
	\end{align*}
	with equality if and only if $m=d$ and the vectors $v_i$ form an orthonormal basis. 
\end{conjecture}

It should be noted that for $p=1$, the vector-valued Maclaurin inequality is of particular interest, as it turns out to be closely related to the far-reaching Aleksandrov--Fenchel inequality from convex geometry. By a simple argument one can deduce the classical Maclaurin inequality as a consequence of Newton's inequality, and similarly one would be able to deduce the vector-valued Maclaurin inequality for $p=1$ from a corresponding vector-valued version of Newton's inequality of the following form
\begin{align}\label{ineq:vec newton}
	\nonumber&\left(\frac{\sum\limits_{1\le i_1 < \cdots < i_{k}\le m}|v_{i_1}\wedge\cdots\wedge v_{i_{k}}|}{\binom{m}{k}}\right)^{2} \ge \\
	 &\left(\frac{\sum\limits_{1\le i_1 < \cdots < i_{k-1}\le m}|v_{i_1}\wedge \cdots\wedge v_{i_{k-1}}|}{\binom{m}{k-1}}\right)\left(\frac{\sum\limits_{1\le i_1 < \cdots < i_{k+1}\le m}|v_{i_1}\wedge \cdots\wedge v_{i_{k+1}}|}{\binom{m}{k+1}}\right),
\end{align}
where $2\le k\le d-1$. It is worth noting the explicit connection between Newton's inequality and the Aleksandrov--Fenchel inequality. Up until this point, in the literature, the Aleksandrov--Fenchel inequality has been referred to as a Newton-type inequality, simply because it is of the same form (square greater than a product).  

To illustrate the connection between the vector-valued Maclaurin inequality and the Aleksandrov--Fenchel inequality, let us denote by $P$ the following Minkowski sum of line segments
\begin{align*}
	P = \sum_{j=1}^{m}\frac{1}{2}[-v_j,v_j].
\end{align*}
Using this notation, Conjecture \ref{conj:vec maclaurin} with $p=1$ exactly states that for $1<k\le d$ we have
\begin{align}\label{ineq:vec maclaurin intrinsic volumes}
	\left(\frac{V_{k}(P)}{\binom{m}{k}}\right)^{\frac{1}{k}}\le \left(\frac{V_{k-1}(P)}{\binom{m}{k-1}}\right)^{\frac{1}{k-1}},
\end{align}
or 
\begin{align*}
	\left(\frac{V_{k}(P)}{V_k(C_m)}\right)^{\frac{1}{k}}\le \left(\frac{V_{k-1}(P)}{V_{k-1}(C_m)}\right)^{\frac{1}{k-1}},
\end{align*}
where $V_k$ denotes the $k$-th intrinsic volume and $C_m$ is the $m$-dimensional unit cube. This is a general isoperimetric-type inequality. For example, the case $m=d$ and and $k=d$, which is proved here, says that among all parallelepipeds (non necessarily orthogonal) with the same volume, the cube has the smallest surface area. This particular case was first proved by Hadwiger in \cite{hadwiger_book} (see also \cite{handbook}) using Steiner symmetrisation.  Moreover, if (\ref{ineq:vec maclaurin intrinsic volumes}) holds for an arbitrary sum of line segments, then one would recover the dimension free estimate of McMullen \cite{mcmullen1991inequalities} restricted to the class of zonoids. Furthermore, (\ref{ineq:vec maclaurin intrinsic volumes}) is related to isoperimetric-type inequalities proved in \cite{schneider_hug}. 

Despite the fact that we don't have a proof for the sharp inequality \eqref{ineq:vec maclaurin intrinsic volumes}, we are able to prove it with a constant that it is bounded by an absolute constant that doesn't depend on the dimension. 
\begin{theorem}\label{non-sharp}
	For any $d$-tuple of vectors $v_1,\dots,v_d \in \R^d$ and any $2 < k \le d$, the following inequality holds:
	\begin{align*}
		\left(\frac{\sum\limits_{1\le i_1 < \cdots < i_{k}\le d}|v_{i_1}\wedge\cdots\wedge v_{i_{k}}|}{\binom{d}{k}}\right)^{\frac{1}{k}} \le \frac{2(d-k+1)}{(d-k+2)} \left(\frac{\sum\limits_{1\le i_1 < \cdots < i_{k-1}\le d}|v_{i_1}\wedge \cdots\wedge v_{i_{k-1}}|}{\binom{d}{k-1}}\right)^{\frac{1}{k-1}}
	\end{align*}
\end{theorem}
\noindent Note that the constant appearing on the right-hand side is greater than $1$, but smaller than 2.

\section*{Structure of paper}
In Section \ref{section:notation and background} we introduce some relevant notation and terminology. Section \ref{section:principal minors and szaszs inequality} is dedicated to proving Theorems \ref{thm:vec maclaurin with squares}, \ref{thm:vec maclaurin with p = infty} and 
\ref{thm:vec maclaurin with p = 0}. In Section \ref{section:partial results for p=1 and monotonicity argument}, we discuss a general approach for attempting to establish (\ref{ineq:vec maclaurin with p}) with $p=1$, and prove the special cases listed in Theorem \ref{thm:vec maclaurin list of known p=1 cases}. At the beginning of Section \ref{section:connections to intrinsic volumes}, we introduce tools from convex geometry which allow us to rewrite the vector-valued Maclaurin inequality with $p=1$ in terms of convex bodies and mixed volumes. In Section \ref{section:nonsharp p=1 case}, we then use these tools to establish Theorem \ref{non-sharp}.

\section{Notation and background information}\label{section:notation and background}
We work in ${\mathbb R}^d$, which is equipped with a Euclidean structure $\langle \cdot,\cdot\rangle$
and we fix an orthonormal basis $\{e_1,\ldots ,e_d\}$. We denote by $B_2^d$ and $S^{d-1}$ the Euclidean unit ball and sphere
in ${\mathbb R}^d$ respectively. We write $\sigma $ for the normalised rotationally invariant probability measure on $S^{d-1}$ and $\nu $
for the Haar probability measure on the orthogonal group $O(d)$. Let $G_{d,k}$ denote the Grassmannian of all $k$-dimensional
subspaces of ${\mathbb R}^d$. Then, $O(d)$ equips $G_{d,k}$ with a Haar probability measure $\nu_{d,k}$.
The letters $c,c^{\prime }, c_1, c_2$ etc. denote absolute positive constants which may change from line to line. Whenever we write
$a\simeq b$, we mean that there exist absolute constants $c_1,c_2>0$ such that $c_1a\leq b\leq c_2a$.

Let ${\mathcal K}_d$ denote the class of all non-empty compact convex subsets of ${\mathbb R}^d$.
If $K\in {\mathcal K}_d$ has non-empty interior, we will say that $K$ is a convex body. If $A\in {\mathcal K}_d$,
we will denote by $|A|$ the volume of $A$ in the appropriate affine subspace unless otherwise stated. The volume of $B_2^d$ is denoted by $\omega_d$.
We say that a convex body $K$ in ${\mathbb R}^d$ is symmetric if $x\in K$ implies that $-x\in K$, and that $K$ is centred
if its centre of mass $\frac{1}{|K|}\int_Kx\,dx $ is at the origin.
The support function of a
convex body $K$ is defined by $h_K(y)=\max \{\langle x,y\rangle :x\in K\}$.
For any $E\in G_{d,k}$ we denote by $E^{\perp }$ the orthogonal subspace of $E$, i.e. $E^{\perp }=\{x\in {\mathbb R}^d:
\langle x,y\rangle =0\;\hbox{for all}\,y\in E\}$. In particular, for any $u\in S^{d-1}$ we define $u^{\perp }=\{x\in {\mathbb R}^d:\langle x,u\rangle =0\}$.
The section of $K\in {\mathcal K}_d$ with a subspace $E$ of ${\mathbb R}^d$ is $K\cap E$, and the orthogonal projection of $K$ onto $E$ is denoted
by $P_E(K)$.

Mixed volumes are introduced by a classical theorem of Minkowski which describes the way volume behaves with respect
to the operations of addition and multiplication of compact convex sets by non-negative reals: if $K_1,\ldots ,K_N\in {\cal K}_d$, $N\in {\mathbb N}$,
then the volume of $t_1K_1+\cdots +t_NK_N$ is a homogeneous polynomial of degree $d$ in $t_i\geq 0$ (see \cite{Burago-Zalgaller-book} and \cite{Schneider-book}):
\begin{equation}\label{eq:not-2}\big |t_1K_1+\cdots +t_NK_N\big |=\sum_{1\leq i_1,\ldots ,i_d\leq N} V(K_{i_1},\ldots
	,K_{i_d})t_{i_1}\ldots t_{i_d},\end{equation}
where the coefficients $V(K_{i_1},\ldots ,K_{i_d})$ are chosen to be invariant under permutations of their
arguments. The coefficient $V(K_{i_1},\ldots ,K_{i_d})$ is called the mixed volume of the $d$-tuple $(K_{i_1},\ldots ,K_{i_d})$.
We will often use the fact that $V$ is positive linear with respect to each of its arguments and that
$V(K,\ldots ,K)=|K|$ (the $d$-dimensional Lebesgue measure of $K$) for all $K\in {\cal K}_d$.

Steiner's formula is a special case of Minkowski's theorem. If $K\in {\cal K}_d$ then the volume of $K+tB_2^d$, $t>0$,
can be expanded as a polynomial in $t$:
\begin{equation}\label{eq:not-3}|K+tB_2^d|=\sum_{k=0}^d\binom{d}{k}W_k(K)t^k,\end{equation}
\noindent where $W_k(K):=V(K[d-k],B_2^d[k])$ is the $k$-th quermassintegral of $K$.
Moreover, for $k =1,\dots,d$, the $k$-th intrinsic volume of a convex body $L\subset \R^d$ is defined as
\begin{align*}
	V_k(L) = \binom{d}{k}\frac{V(L[k],B_2^d[d-k])}{\omega_{d-k}}.
\end{align*}
The Aleksandrov-Fenchel inequality states that if $K,L,K_3,\ldots ,K_d\in {\cal K}_d$, then
\begin{equation}\label{eq:not-4}V(K,L,K_3,\ldots ,K_d)^2\geq V(K,K,K_3,\ldots ,K_d) V(L,L,K_3,\ldots ,K_d).\end{equation}
In particular, this implies that the sequence $(W_0(K),\ldots ,W_d(K))$ is
log-concave. From the Aleksandrov-Fenchel inequality one can
recover the Brunn-Minkowski inequality as well as the following generalisation
for the quermassintegrals:
\begin{equation}\label{eq:not-6}W_k(K+L)^{\frac{1}{d-k}}\geq W_k(K)^{\frac{1}{d-k}}+W_k(L)^{\frac{1}{d-k}},\qquad k=0,\ldots ,d-1.\end{equation}
We write $S(K)$ for the surface area of $K$. From Steiner's formula and the definition of surface area we see that $S(K)=dW_1(K)$.
Finally, let us mention Kubota's integral formula
\begin{equation}\label{eq:not-7}W_k(K)=\frac{\omega_d}{\omega_{d-k}}\int_{G_{d,d-k}}
	|P_E(K)|\,d\nu_{d,d-k}(E),\qquad 1\leq k\leq d-1.\end{equation}
The case $k=1$ is Cauchy's surface area formula
\begin{equation}\label{eq:not-8}S(K)=\frac{\omega_d}{d\omega_{d-1}}\int_{S^{d-1}}|P_{u^{\perp }}(K)|\,d\sigma (u).\end{equation}
We refer to the books \cite{Gardner-book} and \cite{Schneider-book} for basic facts from the Brunn-Minkowski theory and to the books
\cite{AGA-book} and \cite{BGVV-book} for basic facts from asymptotic convex geometry.

\section{Principal minors and Szasz's inequality}\label{section:principal minors and szaszs inequality}
In this section we introduce some elementary tools from linear algebra, and use them to prove Theorems \ref{thm:vec maclaurin with squares}, \ref{thm:vec maclaurin with p = infty} and \ref{thm:vec maclaurin with p = 0}.

\paragraph{Vector-valued Maclaurin inequality with $p=2$}
Firstly, for any family of vectors $v_1\dots,v_d\in\R^d$, we write the square of each $k$-dimensional volume $|v_{i_1}\wedge \cdots v_{i_k}|^2$ as the determinant of a $k \times k$ submatrix of some fixed $d \times d$ matrix. Let us introduce the notion of a principal minor:
\begin{definition}[Principal minors]
	Let $M$ be an $n \times n$ matrix. For $S \subset [n]$, define $M_S$ to be submatrix constructed by removing rows and edges with indices not in $S$. The set of principal submatrices of $M$ is defined as $\{M_S \, |\, S\subset [n] \}$. Furthermore, the set of principal minors of $M$ is defined as $\{\det\left(M_S\right) \, |\, S\subset [n] \}$. For $1\le k\le n$, we define the principal $k$-submatrices and principal $k$-minors of $M$ by adding the condition that $|S| = k$.
\end{definition}
\noindent Let $A$ denote the square $d\times d$ matrix with columns $v_1,\dots,v_d$ and let $B$ be the $d \times k$ matrix with columns $v_1,\dots,v_k$ for some $1\le k\le d$, then we can write
\begin{align*}
	|v_1 \wedge \cdots \wedge v_k|^2 = \det(B^{T} B).
\end{align*}
A short proof of this identity can be found in \cite{vol_para_and_zonoids}. The matrix $B^{T} B$ is a principal $k$-submatrix of $A^{T}A$, and can be constructed by removing the last $(d-k)$ rows and columns. In this way we see that the sum of terms $|v_{i_1}\wedge \cdots \wedge v_{i_k}|^2$ over all $1\le i_1<\cdots <i_k\le d$ is equal to the sum of all principal $k$-minors of $A^{T}A$. The next lemma allows us to work with the sum of all principal minors of $A^TA$:
\begin{lemma}[Sum of principal minors]\label{lem:sum of principal minors equals symmetric poly of eigenvalues}
	Let $M$ be a $n \times n$ matrix with eigenvalues $\lambda_1,\dots,\lambda_n$ (not necessarily distinct). For $1\le k\le n$ we have that
	\begin{align*}
		\sum_{|S| = k}\det(M_S^{T}M_S) = \sum_{|S| = k} \prod_{i \in S}\lambda_{i} .
	\end{align*}
\end{lemma}
\noindent A proof of this lemma can be found in \cite{meyer2000matrix}. Now we are ready to establish the vector-valued Maclaurin inequality with $p=2$:
\medskip

\begin{proof}[Proof of Theorem \ref{thm:vec maclaurin with squares}]
	As before, let $A$ denote the square matrix with columns $v_1,\dots,v_d$ and (not necessarily distinct) eigenvalues $\lambda_1,\dots,\lambda_d$.  Then for $1\le i_1<\cdots< i_k\le d$, define $B_{i_1,\dots,i_k}$ to be the $d \times k$ matrix with columns $v_{i_1},\dots,v_{i_k}$. Now we use Lemma \ref{lem:sum of principal minors equals symmetric poly of eigenvalues} along with Theorem \ref{thm:maclaurin inequalities}, which is precisely the classical Maclaurin inequality, to write
	\begin{align*}
		\left(\frac{\sum\limits_{1\le i_1 < \cdots < i_{k}\le d} |v_{i_1}\wedge \cdots\wedge v_{i_{k}}|^2 }{\binom{d}{k}}\right)^{\frac{1}{k}}&= \left(\frac{\sum\limits_{1\le i_1 < \cdots < i_{k}\le d} \det\left(B_{i_1,\dots,i_{k}}^{T}B_{i_1,\dots,i_{k}} \right)}{\binom{d}{k}}\right)^{\frac{1}{k}}\\ &= \left(\frac{\sum\limits_{1\le i_1 < \cdots < i_{k}\le d} \lambda_{i_1}\cdots\lambda_{i_{k}} }{\binom{d}{k}}\right)^{\frac{1}{k}} \\ &\le \left(\frac{\sum\limits_{1\le i_1 < \cdots < i_{k-1}\le d} \lambda_{i_1}\cdots\lambda_{i_{k-1}} }{\binom{d}{k-1}}\right)^{\frac{1}{k-1}}\\ &= \left(\frac{\sum\limits_{1\le i_1 < \cdots < i_{k-1}\le d} \det\left(B_{i_1,\dots,i_{k-1}}^{T}B_{i_1,\dots,i_{k-1}} \right) }{\binom{d}{k-1}}\right)^{\frac{1}{k-1}}\\ &= \left(\frac{\sum\limits_{1\le i_1 < \cdots < i_{k-1}\le d}|v_{i_1}\wedge \cdots\wedge v_{i_{k-1}}|^2}{\binom{d}{k-1}}\right)^{\frac{1}{k-1}}.
	\end{align*}
	This proves the desired result. 
\end{proof}

\paragraph{Endpoint cases $p=0$ and $p=\infty$}
To begin, let us state a result of Szasz regarding principal minors:
\begin{lemma}[Szasz's inequality]\label{ineq:Szasz}
	Let $M$ be some $n \times n$ matrix. For $1< k < n$ the following inequality holds
	\begin{align*}
		\left(\prod_{|A| = k} \det(M_A) \right)^{\frac{1}{\binom{n-1}{k-1}}} \le \left(\prod_{|B| = k-1} \det(M_B) \right)^{\frac{1}{\binom{n-1}{k-2}}} .
	\end{align*}
\end{lemma}
A proof of this using elementary methods can be found in \cite{vol_para_and_zonoids}. Simple applications of this lemma allow us to deduce the vector-valued Maclaurin inequality with $p=0$ and $p=\infty$.
\begin{proof}[Proof of Theorem \ref{thm:vec maclaurin with p = 0}.]
	Let $A$ denote the $d \times m$ matrix with columns $v_1,\dots,v_m$ and let $A_{i_1,\dots,i_k}$ be the $d \times k$ matrix with colums $v_{i_1},\dots,v_{i_k}$ for some $1\le k\le d$, then we have
	\begin{align*}
		|v_{i_1} \wedge \cdots \wedge v_{i_k}|^2 = \det(A_{i_1,\dots,i_k}^{T} A_{i_1,\dots,i_k}).
	\end{align*}
	As we mentioned earlier, $A_{i_1,\dots,i_k}^{T} A_{i_1,\dots,i_k}$ can also be seen as a principal submatrix of $A^TA$, which is an $m\times m$ matrix. Hence, by Szasz's lemma we have
	\begin{align*}
		\left(\prod\limits_{1\le i_1 < \cdots < i_{k}\le m} |v_{i_1}\wedge \cdots\wedge v_{i_{k}}|^2 \right)^{\frac{1}{\binom{m-1}{k-1}}}\le \left(\prod\limits_{1\le i_1 < \cdots < i_{k-1}\le m} |v_{i_1}\wedge \cdots\wedge v_{i_{k-1}}|^2 \right)^{\frac{1}{\binom{m-1}{k-2}}}.
	\end{align*}
	Taking square roots, and noting that $\binom{m-1}{k-1} = \binom{m}{k} \cdot \frac{k}{m}$ and 
	$\binom{m-1}{k-2} = \binom{m}{k-1} \cdot \frac{k-1}{m}$, the above inequality can be written as
	\begin{align}\label{ineq:vec maclaurin geometric mean instead of arithmetic mean}
		\left(\left(\prod\limits_{1\le i_1 < \cdots < i_{k}\le m} |v_{i_1}\wedge \cdots\wedge v_{i_{k}}| \right)^{\frac{1}{\binom{m}{k}}}\right)^{\frac{1}{k}}\le \left(\left(\prod\limits_{1\le i_1 < \cdots < i_{k-1}\le m} |v_{i_1}\wedge \cdots\wedge v_{i_{k-1}}| \right)^{\frac{1}{\binom{m}{k-1}}}\right)^{\frac{1}{k-1}},
	\end{align}
as required.
\end{proof}
\begin{remark}
In \eqref{ineq:vec maclaurin geometric mean instead of arithmetic mean} we see that we have geometric means appearing inside the first set of parenthesis on both sides. Intriguingly, simply replacing these geometric means by arithmetic means reveals the vector valued Maclaurin inequality with $p=1$. So, if we think of the vector-valued Maclaurin inequality with $p=1$ as a chain of inequalities for a sequence of arithmetic means, then (\ref{ineq:vec maclaurin geometric mean instead of arithmetic mean}) can be thought of as an analogous chain of inequalities for the corresponding geometric means.
\end{remark}
The case for $p=\infty$ also follows directly from Szasz's inequality.
\begin{proof}[Proof of Theorem \ref{thm:vec maclaurin with p = infty}]
	Let $i_1<\cdots<i_k\leq m$ be the indices where the left hand side is maximised and let $B$ be the matrix with columns $v_{i_1},\ldots, v_{i_k}$. Now, using Szasz's inequality for $M = B^{T}B$ and $n=k$ we get
	\begin{align*}
		|v_{i_1}\wedge\cdots\wedge v_{i_{k}}| &\le \left(\prod_{j = 1}^k |v_{i_1}\wedge\cdots\wedge \widehat{v_{i_j}}\wedge\cdots\wedge v_{i_{k}}| \right)^{\frac{1}{k-1}}\\ &\le \left(\prod_{j = 1}^k \left(\max\limits_{1\le i_1 < \cdots < i_{k-1}\le d}|v_{i_1}\wedge\cdots\wedge v_{i_{k-1}}|\right) \right)^{\frac{1}{k-1}} \\&= \left(\max\limits_{1\le i_1 < \cdots < i_{k-1}\le d}|v_{i_1}\wedge\cdots\wedge v_{i_{k-1}}|\right)^{\frac{k}{k-1}},
	\end{align*}
	which concludes the proof.
\end{proof}

\section{Partial results for $p=1$ and monotonicity argument}\label{section:partial results for p=1 and monotonicity argument}
In the next section we prove the special cases of vector-valued Maclaurin inequalities with $p=1$ and $m=d$ listed in Theorem \ref{thm:vec maclaurin list of known p=1 cases}. Our method is somewhat inspired by a monotonicity argument given in \cite{inequalities} to prove the classical Maclaurin inequality. Given vectors $v_1,\dots,v_d \in \R^d$, we attempt to construct a second family of orthogonal vectors $\tilde{v}_1,\dots,\tilde{v}_d \in \R^d$, such that
\begin{align*}
	\left(\frac{\sum\limits_{1\le i_1 < \cdots < i_{k}\le d}|v_{i_1}\wedge\cdots\wedge v_{i_{k}}|}{\binom{d}{k}}\right)^{\frac{1}{k}} \le
	\left(\frac{\sum\limits_{1\le i_1 < \cdots < i_{k}\le d}|\tilde{v}_{i_1}\wedge\cdots\wedge \tilde{v}_{i_{k}}|}{\binom{d}{k}}\right)^{\frac{1}{k}},
\end{align*}
and
\begin{align*}
	\left(\frac{\sum\limits_{1\le i_1 < \cdots < i_{k-1}\le d}|v_{i_1}\wedge \cdots\wedge v_{i_{k-1}}|}{\binom{d}{k-1}}\right)^{\frac{1}{k-1}} \ge
	\left(\frac{\sum\limits_{1\le i_1 < \cdots < i_{k-1}\le d}|\tilde{v}_{i_1}\wedge \cdots\wedge \tilde{v}_{i_{k-1}}|}{\binom{d}{k-1}}\right)^{\frac{1}{k-1}}.
\end{align*}
If such an orthogonal family exists, then applying Theorem \ref{thm:maclaurin inequalities} with positive numbers $\|v_1\|,\dots,\|v_d\|$, we can write
\begin{align*}
	\left(\frac{\sum\limits_{1\le i_1 < \cdots < i_{k}\le d}|v_{i_1}\wedge\cdots\wedge v_{i_{k}}|}{\binom{d}{k}}\right)^{\frac{1}{k}}
	&\le
	\left(\frac{\sum\limits_{1\le i_1 < \cdots < i_{k}\le d}|\tilde{v}_{i_1}\wedge\cdots\wedge \tilde{v}_{i_{k}}|}{\binom{d}{k}}\right)^{\frac{1}{k}}\\ 
	&=
	\left(\frac{\sum\limits_{1\le i_1 < \cdots < i_{k}\le d}\|\tilde{v}_{i_1}\|\cdots\| \tilde{v}_{i_{k}}\|}{\binom{d}{k}}\right)^{\frac{1}{k}}\\ 
	&\le
	\left(\frac{\sum\limits_{1\le i_1 < \cdots < i_{k-1}\le d}\|\tilde{v}_{i_1}\|\cdots\| \tilde{v}_{i_{k-1}}\|}{\binom{d}{k-1}}\right)^{\frac{1}{k-1}}\\
	&=
		\left(\frac{\sum\limits_{1\le i_1 < \cdots < i_{k-1}\le d}|\tilde{v}_{i_1}\wedge \cdots\wedge \tilde{v}_{i_{k-1}}|}{\binom{d}{k-1}}\right)^{\frac{1}{k-1}}\\
	&\le
		\left(\frac{\sum\limits_{1\le i_1 < \cdots < i_{k-1}\le d}|v_{i_1}\wedge \cdots\wedge v_{i_{k-1}}|}{\binom{d}{k-1}}\right)^{\frac{1}{k-1}}.
\end{align*}
In order to streamline the argument slightly, we introduce the following convenient notation
\begin{align*}
	S_k(v_{1},\dots,v_{d}) := \sum\limits_{\{i_1, \dots ,i_{k}\}\subset[d]}|v_{i_1}\wedge \cdots\wedge v_{i_{k}}|.
\end{align*}
Notice that by symmetry, instead of constructing a whole family of orthogonal vectors, it suffices to construct $\tilde{v}_1 \in \{v_2,\dots,v_d \}^{\perp}$ such that
\begin{align*}
	S_k(v_{1},\dots,v_{d}) \le S_k(\tilde{v}_{1},v_2,\dots,v_{d}),
\end{align*}
and
\begin{align*}
	S_{k-1}(v_{1},\dots,v_{d}) \ge S_{k-1}(\tilde{v}_{1},v_2,\dots,v_{d}).
\end{align*}
Since we only need to worry about the terms involving $v_1$, these last two conditions can be written respectively as
\begin{align*}
	\sum\limits_{\{i_1, \dots ,i_{k-1}\}\subset[d]\setminus 1}|v_1\wedge v_{i_1}\wedge\cdots\wedge v_{i_{k-1}}| \le \sum\limits_{\{i_1, \dots ,i_{k-1}\}\subset[d]\setminus 1}|\tilde{v}_1\wedge v_{i_1}\wedge\cdots\wedge v_{i_{k-1}}|,
\end{align*}
and
\begin{align*}
	\sum\limits_{\{i_1, \dots ,i_{k-2}\}\subset[d]\setminus 1}| v_1\wedge v_{i_1}\wedge\cdots\wedge v_{i_{k-2}}| \ge \sum\limits_{\{i_1, \dots ,i_{k-2}\}\subset[d]\setminus 1}|  \tilde{v}_1\wedge v_{i_1}\wedge\cdots\wedge v_{i_{k-2}}|.
\end{align*}
Note that $\tilde{v}_1$ is orthogonal to the vectors $v_2,\dots,v_d$, so for any $\{v_{i_1},\dots,v_{i_r}\} \subset \{v_2,\dots,v_d\}$
\begin{align*}
	|\tilde{v}_1\wedge v_{i_1}\wedge\cdots\wedge v_{i_r}| = \|\tilde{v}_1\||v_{i_1}\wedge\cdots\wedge v_{i_r}|.
\end{align*}
Hence we can rewrite the previous two inequalities as follows,
\begin{align*}
	\frac{\sum\limits_{\{i_1, \dots ,i_{k-1}\}\subset[d]\setminus 1}|v_1\wedge v_{i_1}\wedge\cdots\wedge v_{i_{k-1}}| }{\sum\limits_{\{i_1, \dots ,i_{k-1}\}\subset[d]\setminus 1}|v_{i_1}\wedge\cdots\wedge v_{i_{k-1}}| } \le \|\tilde{v}_1\| \le \frac{\sum\limits_{\{i_1, \dots ,i_{k-2}\}\subset[d]\setminus 1}| v_1\wedge v_{i_1}\wedge\cdots\wedge v_{i_{k-2}}|}{\sum\limits_{\{i_1, \dots ,i_{k-2}\}\subset[d]\setminus 1}| v_{i_1}\wedge\cdots\wedge v_{i_{k-2}}|}.
\end{align*}
So the question is can choose a length $\|\tilde{v}_1\|$ satisfying the above? Let us summarise what we have just derived with the following result:
\begin{theorem}\label{thm:vec maclaurin p=1 reduction}
	Fix $1<k\le d$. Suppose that for any $v_1,\dots,v_d \in \R^d$
	\begin{align}\label{ineq:vec maclaurin p=1 reduction}
		\frac{\sum\limits_{\{i_1, \dots ,i_{k-1}\}\subset[d]\setminus 1}|v_1\wedge v_{i_1}\wedge\cdots\wedge v_{i_{k-1}}| }{\sum\limits_{\{i_1, \dots ,i_{k-1}\}\subset[d]\setminus 1}|v_{i_1}\wedge\cdots\wedge v_{i_{k-1}}| } \le  \frac{\sum\limits_{\{i_1, \dots ,i_{k-2}\}\subset[d]\setminus 1}| v_1\wedge v_{i_1}\wedge\cdots\wedge v_{i_{k-2}}|}{\sum\limits_{\{i_1, \dots ,i_{k-2}\}\subset[d]\setminus 1}| v_{i_1}\wedge\cdots\wedge v_{i_{k-2}}|},
	\end{align}
	where we interpret the $k=2$ case as
	\begin{align*}
		\frac{\sum\limits_{i \in [d]\setminus 1}| v_1\wedge v_i|}{\sum\limits_{i \in [d]\setminus 1}\|v_i\|} \le \|v_1\|.
	\end{align*}
	Then, for any $\omega_1,\dots,\omega_d \in \R^d$, we have
	\begin{align*}
		\left(\frac{S_k(\omega_1,\dots,\omega_d)}{\binom{d}{k}}\right)^{\frac{1}{k}} \le \left(\frac{S_{k-1}(\omega_1,\dots,\omega_d)}{\binom{d}{k-1}}\right)^{\frac{1}{k-1}}.
	\end{align*}
\end{theorem}

\subsection{Proof of Theorem \ref{thm:vec maclaurin list of known p=1 cases}}
Now we will prove various special cases of (\ref{ineq:vec maclaurin p=1 reduction}) which then imply the corresponding cases listed in Theorem \ref{thm:vec maclaurin list of known p=1 cases}:
\begin{lemma}\label{lem:endpoints k=2 and k=d vector valued maclaurin inequalities using projection method}
	If $k=2$ or $k=d$, then $(\ref{ineq:vec maclaurin p=1 reduction})$ holds for arbitrary $v_1,\dots,v_d \in\R^d$.
\end{lemma}
\begin{proof}
	To begin, let us deal with the case where $k=d$. Observe that for $2\le i \le d$ we have
	\begin{align*}
		\frac{|v_1\wedge \dots \wedge v_d|}{|v_2\wedge \dots \wedge v_{d}|} = \|\pi_{\{v_2,\dots,v_d\}^{\perp}}v_1\| \le \|\pi_{\{v_2,\dots,\widehat{v_i},\dots,v_d\}^{\perp}}v_1\| = \frac{|v_1\wedge \dots\wedge \widehat{v_i} \wedge \dots \wedge v_{d}|}{|v_2\wedge \dots\wedge \widehat{v_i} \wedge \dots \wedge v_{d}|}.
	\end{align*}
	Rearranging this gives
	\begin{align*}
		|v_1\wedge \dots \wedge v_d||v_2\wedge \dots\wedge \widehat{v_i} \wedge \dots \wedge v_{d}| \le |v_1\wedge \dots\wedge \widehat{v_i} \wedge \dots \wedge v_{d}||v_2\wedge \dots \wedge v_{d}|.
	\end{align*}
	Summing over $i$ yields
	\begin{align*}
		\sum_{i=2}^d |v_1\wedge \dots \wedge v_d||v_2\wedge \dots\wedge \widehat{v_i} \wedge \dots \wedge v_{d}| \le  \sum_{i=2}^d |v_1\wedge \dots\wedge \widehat{v_i} \wedge \dots \wedge v_{d}||v_2\wedge \dots \wedge v_{d}|
	\end{align*}
	which implies that
	\begin{align*}
		\frac{|v_1\wedge \dots \wedge v_d|}{|v_2\wedge \dots \wedge v_{d}|} \le \frac{\sum_{i=2}^d |v_1\wedge \dots\wedge \widehat{v_i} \wedge \dots \wedge v_{d}|}{\sum_{i=2}^d|v_2\wedge \dots\wedge \widehat{v_i} \wedge \dots \wedge v_{d}|},
	\end{align*}
	which is exactly $(\ref{ineq:vec maclaurin p=1 reduction})$ for $k=d$. Now for the case where $k=2$. Simply note that
	\begin{align*}
		\sum_{i=2}^d |v_1\wedge v_i| \le \|v_1\| \sum_{i=2}^d \|v_i\|,
	\end{align*}
	which immediately gives
	\begin{align*}
		\frac{\sum_{i=2}^d |v_1\wedge v_i|}{\sum_{i=2}^d \|v_i\|} \le \|v_1\|,
	\end{align*}
which concludes the proof.
\end{proof}
Next we deal with the case for $k=3$, which requires a little more work:
\begin{lemma}\label{lem:k=3 vector valued maclaurin inequalities using projection method}
	If $k=3$, then $(\ref{ineq:vec maclaurin p=1 reduction})$ holds in all dimensions $d$.
\end{lemma}
\noindent In order to prove Lemma \ref{lem:k=3 vector valued maclaurin inequalities using projection method}, we first prove the following variant of (\ref{ineq:vec maclaurin p=1 reduction}):
\begin{lemma}\label{lem:2 to d-2 case for reduced vec maclaurin}
	\begin{align}\label{ineq:2 to d-2 case for reduced vec maclaurin}
		\frac{\sum\limits_{1< i_1 < \cdots < i_{d-2}\le d}|v_1\wedge v_{i_1}\wedge\cdots\wedge v_{i_{d-2}}| }{\sum\limits_{1< i_1 < \cdots < i_{d-2}\le d}|v_{i_1}\wedge\cdots\wedge v_{i_{d-2}}| } \le  \frac{\sum\limits_{1< i \le d}| v_1\wedge v_i|}{\sum\limits_{1< i\le d}\| v_i\|}.
	\end{align}
\end{lemma}
\noindent Of course when $d=4$, (\ref{ineq:2 to d-2 case for reduced vec maclaurin}) is exactly (\ref{ineq:vec maclaurin p=1 reduction}) with $k=3$. 

To prove Lemma \ref{lem:2 to d-2 case for reduced vec maclaurin} we need use an elementary fact regarding barycentric coordinates with respect a simplex. The next result is originally to M\"obius \cite{mobius1827barycentrische}:
\begin{proposition}[Barycentric coordinates with respect to a simplex]\label{prop:identity for point inside the simplex}
		Let $v_1,\dots,v_d \in \R^{d-1}$ be the vertices of a $(d-1)$-simplex $\Delta$. Given a vector $u \in \Delta$, there exists a unique $d$-tuple $(\beta_1,\dots,\beta_{d}) \in \R^{d}$ with $\sum_{j=1}^{d} \beta_j = 1$, satisfying the following identity
	\begin{align*}
		\sum_{j=1}^d\beta_j v_j = u.
	\end{align*}
	Furthermore, we can write such numbers $\beta_1,\dots,\beta_d$ explicitly using the following formula 
	\begin{align*}
		\beta_j = \frac{|(v_1 - u)\wedge\cdots\wedge (v_{j-1} - u)\wedge(v_{j+1} - u)\wedge\cdots\wedge (v_{d} - u)|}{2\rm{Vol}(\Delta)}.
	\end{align*}
\end{proposition}
\begin{proof}[Proof of Lemma \ref{lem:2 to d-2 case for reduced vec maclaurin}]
	By rearranging it suffices to prove
	\begin{align*}
		\left(\sum\limits_{1< i_1 < \cdots < i_{d-2}\le d}|v_1\wedge v_{i_1}\wedge\cdots\wedge v_{i_{d-2}}|\right)  \left(\sum\limits_{1< i\le d}\| v_i\|\right) \le &\left(\sum\limits_{1< i \le d}| v_1\wedge v_i|\right)\\
		&\cdot\left(\sum\limits_{1< i_1 < \cdots < i_{d-2}\le d}|v_{i_1}\wedge\cdots\wedge v_{i_{d-2}}|\right).
	\end{align*}
	For $j = 1,\dots,d-2$, it is always true that
	\begin{align*}
		\frac{|v_1\wedge v_{i_1}\wedge\cdots\wedge v_{i_{d-2}}|}{|v_{i_1}\wedge\cdots\wedge v_{i_{d-2}}|} = \left\| \pi_{\{v_{i_1},\dots,v_{i_{d-2}}\}^{\perp}}v_1\right\|\le \left\|  \pi_{{v_{i_j}}^{\perp}}v_1\right\| = \frac{|v_1\wedge v_{i_j}|}{\left\|v_{i_j}\right\|}.
	\end{align*}
Rearranging this we get
	\begin{align}\label{ineq:trivial bound for 2 to d-2 in reduced vec maclaurin}
		|v_1\wedge v_{i_1}\wedge\cdots\wedge v_{i_{d-2}}|\|v_{i_j}\| \le |v_1 \wedge v_{i_j}||v_{i_1}\wedge\cdots\wedge v_{i_{d-2}}|.
	\end{align}
	Applying this to the terms on the left hand side reduces things to proving that
	\begin{align*}
		\sum\limits_{j=1}^d |v_1\wedge v_2\wedge\cdots\wedge\widehat{v_j}\wedge\cdots\wedge v_d|\|v_j\| \le \sum\limits_{j=1}^d |v_1\wedge v_j|| v_2\wedge\cdots\wedge\widehat{v_j}\wedge\cdots\wedge v_d|.
	\end{align*}
	Without loss of generality we may assume that
	\begin{align*}
		\frac{|v_1\wedge v_2|}{\|v_2\|} \le \cdots \le \frac{|v_1\wedge v_d|}{\|v_d\|},
	\end{align*}
	then bearing this in mind we can apply (\ref{ineq:trivial bound for 2 to d-2 in reduced vec maclaurin}) to each term on the left hand side to get
	\begin{align*}
		\sum\limits_{j=1}^d |v_1\wedge v_2\wedge\cdots\wedge\widehat{v_j}\wedge\cdots\wedge v_d|\|v_j\| \le 	\frac{|v_1\wedge v_2|}{\|v_2\|}  \left( \sum_{j=3}^d |v_2\wedge\cdots\wedge\widehat{v_j}\wedge\cdots\wedge v_d|\|v_j\|  \right)&\\ + \frac{|v_1\wedge v_3|}{\|v_3\|}|v_3\wedge\cdots\wedge v_d|\|v_2\|&.
	\end{align*}
	Supposing we know that
	\begin{align*}
		\sum_{j=3}^d |v_2\wedge\cdots\wedge\widehat{v_j}\wedge\cdots\wedge v_d|\|v_j\| \ge |v_3\wedge\cdots\wedge v_d|\|v_2\|,
	\end{align*}
	then by a simple application of the rearrangement inequality for numbers and our assumption on the sizes of the terms $\frac{|v_1\wedge v_j|}{\|v_j\|}$, we see that
	\begin{align*}
		&\frac{|v_1\wedge v_2|}{\|v_2\|}  \left( \sum_{j=3}^d |v_2\wedge\cdots\wedge\widehat{v_j}\wedge\cdots\wedge v_d|\|v_j\|  \right) + \frac{|v_1\wedge v_3|}{\|v_3\|}|v_3\wedge\cdots\wedge v_d|\|v_2\|\\ &\le \frac{|v_1\wedge v_2|}{\|v_2\|}|v_3\wedge\cdots\wedge v_d|\|v_2\| + \frac{|v_1\wedge v_3|}{\|v_3\|} \left( \sum_{j=3}^d |v_2\wedge\cdots\wedge\widehat{v_j}\wedge\cdots\wedge v_d|\|v_j\|  \right)\\ &\le \sum_{j=2}^d |v_1\wedge v_j||v_2\wedge\cdots\wedge\widehat{v_j}\wedge\cdots\wedge v_d|,
	\end{align*}
	which is what we want. It suffices to prove the following claim:
	\begin{claim}\label{clm:beautiful little lemma in R^d}
		For any $u_1,\dots,u_d \in \R^d$ we have
		\begin{align*}
			\sum_{j=2}^d |u_1\wedge\cdots\wedge\widehat{u_j}\wedge\cdots\wedge u_d|\|u_j\| \ge |u_2\wedge\cdots\wedge u_d| \|u_1\|.
		\end{align*}
	\end{claim}
	\begin{proof}[Proof of Claim \ref{clm:beautiful little lemma in R^d}]
		Supposing that all the vectors lie in a $(d-1)$-dimensional subspace, then it is clear that they must be linearly dependent. In particular, one can find coefficients $\alpha_1,\dots,\alpha_d \in \R$, which are not all equal to zero, such that
		\begin{align}\label{eq:linearly dependent identity}
			\sum_{j=1}^d\alpha_j u_j = \mathbf{0}.
		\end{align}
		By assumption, we have $\sum_{j=1}^d|\alpha_j| \neq 0$, so by setting 
		\begin{align*}
			\beta_j = \frac{|\alpha_j|}{\sum_{j=1}^d|\alpha_j|},
		\end{align*}
		and
		\begin{align*}
			\tilde{u}_j =\textrm{sgn}(\alpha_j) u_j,
		\end{align*}
		we can rewrite (\ref{eq:linearly dependent identity}) as
		\begin{align}\label{eq:linearly dependent identity 2}
			\sum_{j=1}^d\beta_j \tilde{u}_j = \mathbf{0}.
		\end{align}
		By definition, we have
		\begin{align*}
			\sum_{j=1}^d\beta_j = 1,
		\end{align*}
		so applying Proposition \ref{prop:identity for point inside the simplex} we can write
		\begin{align*}
			\beta_j = \frac{|\tilde{u}_1\wedge\cdots\wedge \tilde{u}_{j-1} \wedge \tilde{u}_{j+1}\wedge \cdots\wedge \tilde{u}_d|}{\textrm{Vol}\left(\Delta_{\tilde{u}_1,\dots,\tilde{u}_d}\right)} = \frac{|u_1\wedge\cdots\wedge u_{j-1} \wedge u_{j+1}\wedge \cdots\wedge u_d|}{\textrm{Vol}\left(\Delta_{\tilde{u}_1,\dots,\tilde{u}_d}\right)} ,
		\end{align*}
		where $\Delta_{\tilde{u}_1,\dots,\tilde{u}_d}$ denotes the simplex with vertices $\tilde{u}_1,\dots,\tilde{u}_d$. So, (\ref{eq:linearly dependent identity 2}) becomes
		\begin{align*}
			\sum_{j=1}^d\frac{|u_1\wedge\cdots\wedge u_{j-1} \wedge u_{j+1}\wedge \cdots\wedge u_d|}{\textrm{Vol}\left(\Delta_{\tilde{u}_1,\dots,\tilde{u}_d}\right)} \tilde{u}_j = \mathbf{0},
		\end{align*}
		and then multiplying through by $\rm{Vol}\left(\Delta_{\tilde{u}_1,\dots,\tilde{u}_d}\right)$ we get
		\begin{align*}
			\sum_{j=1}^d|u_1\wedge\cdots\wedge u_{j-1} \wedge u_{j+1}\wedge \cdots\wedge u_d| \tilde{u}_j = \mathbf{0}.
		\end{align*}
		By the reverse triangle inequality, we have
		\begin{align*}
			0 &= \left\|  \sum_{j=1}^d |u_1\wedge\cdots\wedge u_{j-1} \wedge u_{j+1}\wedge \cdots\wedge u_d|\tilde{u}_j \right\| \\
			&\ge |u_2\wedge\cdots\wedge u_d| \|\tilde{u}_1\| - \left\|  \sum_{j=2}^d |u_1\wedge\cdots\wedge u_{j-1} \wedge u_{j+1}\wedge \cdots\wedge u_d|\tilde{u}_j \right\|.
		\end{align*}
		Now by simply rearranging and applying the standard triangle inequality, we see that
		\begin{align*}
			|u_2\wedge\cdots\wedge u_d| \|u_1\|  = |u_2\wedge\cdots\wedge u_d| \|\tilde{u}_1\| &\le \left\|  \sum_{j=2}^d |u_1\wedge\cdots\wedge u_{j-1} \wedge u_{j+1}\wedge \cdots\wedge u_d|\tilde{u}_j \right\|\\
			 &\le 	\sum_{j=2}^d |u_1\wedge\cdots\wedge u_{j-1} \wedge u_{j+1}\wedge \cdots\wedge u_d|\|\tilde{u}_j\|\\ 
			 &= \sum_{i=2}^d |u_1\wedge\cdots\wedge u_{j-1} \wedge u_{j+1}\wedge \cdots\wedge u_d|\|u_j\|.
		\end{align*}
		So it suffices to reduce to the case where all the vectors lie in a $(d-1)$-dimensional subspace. Suppose that the vectors $u_1,\dots,u_d$ do not lie in a $(d-1)$-dimensional subspace. Let us define a new family $\omega_1,\dots,\omega_d$ of vectors from the original family by simply projecting $u_1$ onto the subspace spanned by $u_2,\dots,u_d$ and scaling appropriately. More precisely define
		\begin{align*}
			\omega_1 := \frac{\|u_1\|}{\|\pi u_1\|}u_1,
		\end{align*}
		where $\pi = \pi_{{\rm span}\{u_2,\dots,u_d\}}$ and for $j=2,\dots,d$
		\begin{align*}
			\omega_j := u_j.
		\end{align*}
		Clearly
		\begin{align*}
			|\omega_2\wedge\cdots\wedge \omega_d| \|\omega_1\|=|u_2\wedge\cdots\wedge u_d| \|u_1\|,
		\end{align*}
		so it suffices to prove that for $j=2,\dots,d$
		\begin{align*}
			|\omega_1\wedge\cdots\wedge\widehat{\omega_j}\wedge\cdots\wedge \omega_d|\|\omega_j\|\le |u_1\wedge\cdots\wedge\widehat{u_j}\wedge\cdots\wedge u_d|\|u_j\|,
		\end{align*}
		which follows from the fact that
		\begin{align*}
			\frac{|\pi u_1\wedge u_2\wedge\cdots\wedge\widehat{u_j}\wedge\cdots\wedge u_d|}{\|\pi u_1\|} \le \frac{|u_1\wedge\cdots\wedge\widehat{u_j}\wedge\cdots\wedge u_d|}{\|u_1\|}
		\end{align*}
	\end{proof}
	\noindent This concludes the proof of Lemma \ref{lem:2 to d-2 case for reduced vec maclaurin}.
\end{proof}
\noindent By a simple argument we can now establish $(\ref{ineq:vec maclaurin p=1 reduction})$ for $k=3$ in all dimensions:
\begin{proof}[Proof of Lemma \ref{lem:k=3 vector valued maclaurin inequalities using projection method}]
	Directly applying Lemma \ref{lem:2 to d-2 case for reduced vec maclaurin} with  $d=4$, for any $v_1,\dots,v_4 \in \R^4$ we get
	\begin{align*}
		|v_1\wedge v_2\wedge v_3|\|v_4\| + |v_1\wedge v_2\wedge v_4|\|v_3\| + |v_1\wedge v_3\wedge v_4|\|v_2\| \le |v_2 \wedge v_3||v_1\wedge v_4| + &|v_2 \wedge v_4||v_1\wedge v_3|\\
		&+ |v_3 \wedge v_4||v_1\wedge v_2|.
	\end{align*}
	For higher dimensions simply note that by the $d=4$ case we have
	\begin{align*}
		\sum |v_1\wedge v_a\wedge v_b|\|v_c\| &= \sum_{ \{i,j,k\} \subset [d]} |v_1\wedge v_j\wedge v_k|\|v_i\| + |v_1\wedge v_i\wedge v_k|\|v_j\| + |v_1\wedge v_i\wedge v_j|\|v_k\|\\ &\le \sum_{ \{i,j,k\} \subset [d]} |v_j\wedge v_k||v_1\wedge v_i| + | v_i\wedge v_k||v_1\wedge v_j| + |v_i\wedge v_j||v_1\wedge v_k|\\ &= \sum |v_a\wedge v_b||v_1\wedge v_c|.
	\end{align*}
\end{proof}

\section{Connection to intrinsic volumes}\label{section:connections to intrinsic volumes}
As mentioned in the introduction, the vector-valued Maclaurin inequality with $p=1$ can be rewritten as a sequence of inequalities between intrinsic volumes of certain polytopes. Firstly, let us state a useful formula for calculating the mixed volume of a zonoid:
\begin{theorem}[Theorem 5.3.2 in \cite{Schneider-book}]\label{thm:schneider 5.3.2}
	For $i = 1,\dots, j\le d$ let $Z_j$ be a generalised zonoid with generating measure $\rho_i$ and let $K_1,\dots,K_{d-j} \subset \R^d$ be convex bodies. Then
	\begin{align*}
		&V(Z_1,\dots,Z_j,K_1,\dots,K_{d-j})=\\
		&\frac{2^j(d-j)!}{d!}\int_{\Sp^{d-1}} \cdots \int_{\Sp^{d-1}} \left\vert u_1 \wedge\cdots\wedge u_j\right\vert v^{(d-j)}\left(\pi_{\{u_1,\dots,u_j\}^{\perp}}K_1,\dots,\pi_{\{u_1,\dots,u_j\}^{\perp}}K_{d-j} \right) d\rho_1(u_1) \cdots d\rho_j(u_j),
	\end{align*}
	where $v^{(d-j)}$ denotes the $j$-dimensional mixed volume.
\end{theorem}
\noindent Note that any zonotope $Z = \sum\limits_{i=1}^m \alpha_i [-v_i,v_i]$ has a support function defined by
\begin{align*}
	h_Z(u) = \sum_{i=1}^m \alpha_i\left\vert\left\langle u,v_i\right\rangle\right\vert = \int_{\Sp^{d-1}}\left\vert\left\langle u,v\right\rangle\right\vert \,d\rho(v),
\end{align*}
where $\rho$ is concentrated at $\pm v_i$ and assigns mass $\frac{\alpha_i}{2}$ to each of these points. So by Theorem \ref{thm:schneider 5.3.2}, given zonoids $Z_1 = \sum\limits_{k_1} \alpha_{k_1} [-v_{k_1},v_{k_1}],\dots,Z_j = \sum\limits_{k_j} \alpha_{k_j} [-v_{k_j},v_{k_j}]$ and a fixed convex body $K \subset \R^d$, we have
\begin{align*}
	V(Z_1,\dots, Z_j, K[d-j]) =\frac{2^{j}(d-j) !}{d !} \sum_{k_{1}, \dots, k_{j}} \alpha_{k_{1}} \cdots \alpha_{k_{j}} \left|v_{k_{1}}\wedge \dots \wedge v_{k_{j}}\right| \left\vert\pi_{\{v_{k_{1}}, \dots, v_{k_{j}}\}}^{\perp}K\right\vert
\end{align*}
In particular, if $P$ is a centred parallelotope with edges of lengths $\|v_1\|,\dots,\|v_d\|$ in the directions of  $\frac{v_1}{\|v_1\|},\dots,\frac{v_d}{\|v_d\|}$, then we have
\begin{align*}
	h_P(u) = \sum_{i=1}^d \frac{1}{2}\left\vert\left\langle u,v_i\right\rangle\right\vert.
\end{align*}
\noindent It follows that
\begin{align*}
	V_k(P) = \binom{d}{k}\frac{V(P,k; B_2^d, d-k)}{{\omega_{d-k} }} &= \frac{\frac{2^k(d-k)!}{d!} \sum\limits_{\{i_1,\dots,i_k\} \subset [d] } \left(\frac{1}{2}\right)^{k}|v_{i_1}\wedge\cdots\wedge v_{i_k}| \left\vert\pi_{\{v_{i_1},\dots,v_{i_k} \}^{\perp} } B_2^d\right\vert}{\omega_{d-k}}\\ 
	&= \binom{d}{k} \frac{k!(d-k)!}{d!}\sum\limits_{1\le i_1<\cdots<i_k\le d} |v_{i_1}\wedge\cdots\wedge v_{i_k}|\\ &=\sum\limits_{1\le i_1<\cdots<i_k\le d} |v_{i_1}\wedge\cdots\wedge v_{i_k}|.
\end{align*}

In this language, Conjecture \ref{conj:vec maclaurin} states that for all parallelotopes $P \subset \R^d$ and $1< j\le d$, the following inequality holds
\begin{align}\label{ineq:vec valued Maclaurin in terms of intrinsic volumes}
	\left(\frac{V_j(P)}{\binom{d}{j}}\right)^{\frac{1}{j}} \le 	\left(\frac{V_{j-1}(P)}{\binom{d}{j-1}}\right)^{\frac{1}{j-1}},
\end{align}
with equality if and only if $P$ is a cube. Suppose instead we consider the zonotope $Z = \sum\limits_{i=1}^m \frac{1}{2} [-v_i,v_i] \subset \R^d$, then by the same argument we have
\begin{align*}
	V_k(Z) = \binom{d}{k}\frac{V(Z,k; B_2^d, d-k)}{{\omega_{d-k} }} &= \frac{\frac{2^k(d-k)!}{d!} \sum\limits_{\{i_1,\dots,i_k\} \subset [m] } \left(\frac{1}{2}\right)^{k}|v_{i_1}\wedge\cdots\wedge v_{i_k}| \left\vert\pi_{\{v_{i_1},\dots,v_{i_k} \}^{\perp} } B_2^d\right\vert}{\omega_{d-k}}\\ 
	&= \binom{d}{k} \frac{k!(d-k)!}{d!}\sum\limits_{1\le i_1<\cdots<i_k\le m} |v_{i_1}\wedge\cdots\wedge v_{i_k}|\\ &=\sum\limits_{1\le i_1<\cdots<i_k\le m} |v_{i_1}\wedge\cdots\wedge v_{i_k}|.
\end{align*}
So for $k \le d$ we can also write Conjecture \ref{conj:vec maclaurin} as
\begin{align}\label{ineq:vec valued Maclaurin in terms of intrinsic volumes with m}
	\left(\frac{V_j(Z)}{\binom{m}{j}}\right)^{\frac{1}{j}} \le 	\left(\frac{V_{j-1}(Z)}{\binom{m}{j-1}}\right)^{\frac{1}{j-1}}.
\end{align}
Rearranging this gives
\begin{align}\label{ineq:vec valued Maclaurin in terms of intrinsic volumes with m ratio form}
	\frac{V_{j-1}(Z)^j}{V_{j}(Z)^{j-1}} \le \frac{\binom{m}{j-1}^j}{\binom{m}{j}^{j-1}}.
\end{align}
This can be compared now to the Aleksandrov inequalitites for
quermassintegrals:

\begin{equation*}\left (\frac{W_i(K)}{|B_2^n|}\right
	)^{\frac{1}{n-i}} \geq\left (\frac{W_j(K)}{|B_2^n|}\right
	)^{\frac{1}{n-j}}\, \qquad n>i>j\geq 0.
\end{equation*}
Taking into account that $V_j(K) = \binom{n}{j} \frac{W_{n-j}(K)}{\kappa_{n-j}}$ we can rewrite the last one as
\begin{align}\label{ineq:intrinsic volume j to j-1 ratio bound from classical Brunn-Minkowski theory}
	\frac{V_{j-1}(K)^j}{V_{j}(K)^{j-1}} \le \frac{V_{j-1}(B_2^d)^j}{V_{j}(B_2^d)^{j-1}}.
\end{align}
Using the Aleksandrov inequality, McMullen proved the following dimension-free bound for the intrinsic volumes. Namely,
$$V_j(K)^2\geq\frac{j+1}{j}V_{j+1}(K)V_{j-1}(K).$$
Suppose now that Conjecture \ref{conj:vec maclaurin} is true. Then, we will show that we can get a stronger McMullen-type inequality for zonotopes. We can use (\ref{ineq:vec valued Maclaurin in terms of intrinsic volumes with m ratio form}) to recover the above. Indeed, (\ref{ineq:vec valued Maclaurin in terms of intrinsic volumes with m ratio form}) implies log-concavity,
$$V_j(Z)^2\geq\frac{\binom{m}{j}^2}{\binom{m}{j-1}\binom{m}{j+1}}V_{j+1}(Z)V_{j-1}(Z).$$
This can be simplified to
$$V_j(Z)^2\geq\frac{(j+1)(m-j+1)}{j(m-j)}V_{j+1}(Z)V_{j-1}(Z).$$
The factor that appears in the right-hand side is always at least $\frac{j+1}{j}$, which implies McMullen's inequality. Moreover, we attain this bound as $m \to \infty$, as expected.

Note also that in \cite{schneider_hug} the following inequalities were proved: for zonoids $Z$
\begin{align}\label{ineq:hug and schneider reverse ineqaulities}
	\sup\limits_{\Lambda \in GL(d)}\frac{V_j(\Lambda Z)}{V_1(\Lambda Z)^j} \ge \frac{1}{d^j}\binom{d}{j}
\end{align}
if $\dim(Z) = d$ and $j\ge 2$, with equality if and only if $Z$ is a parallelotope. These are reverse, in a sense, to a consequence of (\ref{ineq:intrinsic volume j to j-1 ratio bound from classical Brunn-Minkowski theory}), namely for any convex body $K$
\begin{align}\label{ineq:intrinsic volume 1 to j ratio bound from classical Brunn-Minkowski theory}
	\frac{V_j(K)}{V_1(K)^j} \le \frac{V_j(B_2^d)}{V_1(B_2^d)^j}.
\end{align}
The equality case in (\ref{ineq:hug and schneider reverse ineqaulities}) exactly says that
\begin{align}\label{ineq:hug and schneider reverese equality case}
	\left(\frac{V_j(P)}{\binom{d}{j}}\right)^{\frac{1}{j}} \le 	\frac{V_{1}(P)}{\binom{d}{1}}.
\end{align}
This is of course a special case of (\ref{ineq:vec valued Maclaurin in terms of intrinsic volumes}). If this more general equality case can be established, one might hope to prove the corresponding generalisation to (\ref{ineq:hug and schneider reverse ineqaulities}).

\subsection{A non-sharp vector-valued Maclaurin inequality with $p=1$}\label{section:nonsharp p=1 case}
Let us see how we can use this language borrowed from convex geometry to reinterpret the reduction described at the beginning of Section \ref{section:partial results for p=1 and monotonicity argument}. In particular, let us try to rewrite (\ref{ineq:vec maclaurin p=1 reduction}) purely in terms of intrinsic volumes. Firstly, fix $v_1,\dots,v_{m-1} \in \R^{m-1} \subset \R^{m}$ and $u \in \R^{m-1} \times \R \cong \R^{m}$ , then setting $Z = \sum_{i=1}^{m-1} \frac{1}{2}[-v_i,v_i] \subset \R^{m-1}$, we can apply Theorem \ref{thm:schneider 5.3.2} to get
\begin{align*}
	V_{k}(\pi_{u^{\perp}}Z) &= \binom{m-1}{k}\frac{V\left(\pi_{u^{\perp}}Z,k ; B_2^{m-1},m-k-1\right)}{\omega_{m-k-1}}\\&=\binom{m}{m-1}\binom{m-1}{k}\frac{V\left(Z,k ; B_2^{m},m-k-1; \frac{1}{2}[-u,u]\right)}{\omega_{m-k-1}}\\&=\frac{(m)!}{(m-k-1)!k!}\frac{\left(\frac{2^{k+1}(m-k-1)!}{m!}\right)\sum\limits_{\{i_1,\dots,i_k\} \subset [m-1] } \left(\frac{1}{2}\right)^{k+1}|v_{i_1}\wedge\cdots\wedge v_{i_{k}}\wedge u|\left\vert\pi_{\{v_{i_1},\dots,v_{i_k},u\}^{\perp}}B_2^{m}\right\vert }{\omega_{m-k-1}}\\&=\sum\limits_{1\le i_1<\dots<i_k\le m -1}|v_{i_1}\wedge\cdots\wedge v_{i_{k}}\wedge u|.
\end{align*}
For any $u_1,\dots,u_m \in\R^m$, if we set $Z = \sum_{i=2}^m\frac{1}{2}[-u_i,u_i]$, then using the calculation we have just made along with Theorem \ref{thm:schneider 5.3.2}, inequality (\ref{ineq:vec maclaurin p=1 reduction}) can be rewritten as
\begin{align}\label{ineq:vec maclaurin p=1 reduction in terms of projections}
	\frac{V_{k-1}\left(\pi_{u_1^{\perp}}Z \right)}{V_{k-1}\left(Z \right)} \le \frac{V_{k-2}\left(\pi_{u_1^{\perp}}Z \right)}{V_{k-2}\left(Z \right)}
\end{align}
for $1< k \le m-1$. Theorem 1.2 in \cite{local_A-F_for_zonoids} implies the following result: let $K\subset \R^d$ be a convex body, then for any $u\in \R^d$ we have
\begin{align}\label{ineq:LHS of Theorem 1.2 in local A-F paper}
	\frac{V_{k-1}\left(\pi_{u^{\perp}}K \right)}{V_{k-1}\left(K \right)} \le \frac{2(d-k+1)}{d-k+2} \frac{V_{k-2}\left(\pi_{u^{\perp}}K \right)}{V_{k-2}\left(K \right)},
\end{align}
for all $3\le k\le d$. In the special case where $m= d+1$ and $u_1$ lies in the span of $u_2,\dots,u_m$, one can view  (\ref{ineq:vec maclaurin p=1 reduction in terms of projections}) as the sharp version of (\ref{ineq:LHS of Theorem 1.2 in local A-F paper}) when we restrict ourselves to zonotopes. Using an analogous derivation to the one given at the beginning of Section \ref{section:partial results for p=1 and monotonicity argument}, we can deduce the following result as a consequence of (\ref{ineq:LHS of Theorem 1.2 in local A-F paper}):
\begin{theorem}\label{thm:vec maclaurin with non-sharp constant}
	For any $d$-tuple of vectors $v_1,\dots,v_d \in \R^d$ and any $2 < k \le d$, the following inequality holds:
	\begin{align*}
		\left(\frac{\sum\limits_{1\le i_1 < \cdots < i_{k}\le d}|v_{i_1}\wedge\cdots\wedge v_{i_{k}}|}{\binom{d}{k}}\right)^{\frac{1}{k}} \le \frac{2(d-k+1)}{(d-k+2)} \left(\frac{\sum\limits_{1\le i_1 < \cdots < i_{k-1}\le d}|v_{i_1}\wedge \cdots\wedge v_{i_{k-1}}|}{\binom{d}{k-1}}\right)^{\frac{1}{k-1}}
	\end{align*}
\end{theorem}
\noindent Note that the constant appearing on the right-hand side is greater than $1$, but smaller than $2$. As a consequence of Theorem \ref{thm:vec maclaurin list of known p=1 cases}, we know that the sharp constant is equal to $1$ in dimensions $d=3,4$. Thus, it seems likely that the constant given in Theorem \ref{thm:vec maclaurin with non-sharp constant} is suboptimal.

\bigskip

\noindent {\bf Acknowledgements.} The authors would like to thank Anthony Carbery and Apostolos Giannopoulos for useful discussions. The first named is supported by the Hellenic Foundation for Research and Innovation (Project Number: 1849).

\bigskip

\footnotesize
\bibliographystyle{amsplain}

\begin{thebibliography}{100}
	\footnotesize
	
	\bibitem{AGA-book}
	\textrm{S.\ Artstein-Avidan, A.\ Giannopoulos and V.\ D.\ Milman},
	\textit{Asymptotic Geometric Analysis, Part I}, Amer. Math. Soc., Mathematical
	Surveys and Monographs {\bf 202}  (2015).
	\bibitem{BGVV-book}
	\textrm{S.\ Brazitikos, A.\ Giannopoulos, P.\ Valettas and B-H.\ Vritsiou},
	\textit{Geometry of isotropic convex bodies}, Amer. Math. Society, Mathematical Surveys and Monographs {\bf 196} (2014).
	\bibitem{Burago-Zalgaller-book} {\rm Y.\ D.\ Burago and V.\ A.\ Zalgaller}, \textit{Geometric Inequalities}, Springer Series in Soviet Mathematics,
	Springer-Verlag, Berlin-New York (1988).
	\bibitem{Gardner-book} {\rm R.\ J.\ Gardner}, \textit{Geometric Tomography}, Second Edition, Encyclopedia of Mathematics
	and its Applications {\bf 58}, Cambridge University Press (2006).
	\bibitem{local_A-F_for_zonoids} {\rm A.\ Giannopoulos, M.\ Hartzoulaki and G.\ Paouris},
	\textit{On a local version of the Aleksandrov-Fenchel inequality for the quermassintegrals of a convex body}, 
	Proc. Amer. Math. Soc. {\bf 130} (2002), no. 8, 2403--2412.
	\bibitem{handbook} {\rm P. M. Gruber, J. M. Wills, , G. M. Ziegler.} Handbook of convex geometry. Jahresbericht der Deutschen Mathematiker Vereinigung {\bf 98} (1996), no. 4, 40--40.
	
	\bibitem{hadwiger_book}{\rm H. Hadwiger
	}, \textit{Vorlesungen Über Inhalt, Oberfläche und Isoperimetrie}, Die Grundlehren der Mathematischen Wissenschaften book series {\bf 93}, Springer-Verlag Berlin Heidelberg 1975.
	
	\bibitem{inequalities} {\rm G. H.\ Hardy, J. E.\ Littlewood and G.\ P\'{o}lya}, \textit{Inequalities}, 2nd ed. London: Cambridge University Press (1964).
	\bibitem{schneider_hug} {\rm D.\ Hug and R.\ Schneider}, {\sl Reverse inequalities for zonoids and their application},
	Adv. Math. {\bf 228} (2011), no. 5, 2634--2646.
	\bibitem{vol_para_and_zonoids} {\rm E.\ Gover and N.\ Krikorian},
	\textit{Determinants and the volumes of parallelotopes and zonotopes}, Linear Algebra Appl. {\bf 433} (2010), no. 1, 28--40.
	\bibitem{mcmullen1991inequalities} {\rm P.\ McMullen}, \textit{Inequalities between intrinsic volumes}, Monatsh. Math. {\bf 111} (1991), no. 1, 47--53.
	\bibitem{meyer2000matrix} {\rm C.\ Meyer}, \textit{Matrix analysis and applied linear algebra}, 
	Society for Industrial and Applied Mathematics (SIAM), Philadelphia, PA (2000).
	
	\bibitem{finlays_thesis} {\rm F. D.\ McIntyre}, \textit{Multilinear Geometric Inequalities}, PhD thesis in
	preparation, University of Edinburgh, 2021.
	
	\bibitem{mobius1827barycentrische}
	\textrm{A. \ F. \ M{\"o}bius}
	\textit{Der barycentrische Calcul ein neues H{\"u}lfsmittel zur analytischen Behandlung der Geometrie dargestellt und insbesondere auf die Bildung neuer Classen von Aufgaben und die Entwickelung mehrerer Eigenschaften der Kegelschnitte angewendet von August Ferdinand Mobius Professor der Astronomie zu Leipzig}, Verlag von Johann Ambrosius Barth (1827)
	
	\bibitem{newton1761arithmetica} {\rm I.\ Newton}, \textit{Arithmetica universalis sive de compositione et resolutione arithmetica liber},
	vol. 1, apud Marcum Michaelem Rey (1761).
	\bibitem{Schneider-book} {\rm R.\ Schneider}, \textit{Convex Bodies: The Brunn-Minkowski Theory},
	Second expanded edition. Encyclopedia of Mathematics and Its Applications 151, Cambridge University Press, Cambridge (2014).
\end{thebibliography}

\bigskip

\medskip

\thanks{\noindent {\bf Keywords:}  Maclaurin inequality, andov-Fenchel inequality, mixed volumes, parallelotopes.}

\smallskip

\thanks{\noindent {\bf 2010 MSC:} Primary 52A20; Secondary 52A39, 15A45.}

\bigskip

\bigskip

\noindent \textsc{Silouanos Brazitikos}: Department of
Mathematics, National and Kapodistrian University of Athens, Panepistimiopolis 157-84,
Athens, Greece.

\smallskip

\noindent \textit{E-mail:} \texttt{silouanb@math.uoa.gr}

\bigskip

\noindent \textsc{Finlay McIntyre}: School of Mathematics and Maxwell Institute for Mathematical
Sciences, University of Edinburgh, JCMB, Peter Guthrie Tait Road King's Buildings,
Mayfield Road, Edinburgh, EH9 3FD, Scotland.

\smallskip

\noindent \textit{E-mail:} \texttt{s1204774@sms.ed.ac.uk}

\end{document}